\documentclass[12pt]{amsart}
\usepackage{latexsym, amssymb, amsmath,mathtools}
\usepackage{color}
\usepackage[normalem]{ulem}

\newtheorem{thm}{Theorem}[section]
\newtheorem{lem}[thm]{Lemma}
\newtheorem{cor}[thm]{Corollary}
\newtheorem{prop}[thm]{Proposition}
\newtheorem{rem}[thm]{Remark}
\theoremstyle{definition}
\newtheorem{defn}[thm]{Definition}

\begin{document}

\title[Conformal solitons]
{Classification of conformal solitons in pseudo-Euclidean spaces}


\author{Burcu Bekta\c{s} Demirci}
\address{Fatih Sultan Mehmet Vak{\i}f University, Hali\c{c} Campus, Faculty of Engineering,
Department of Civil Engineering, 34445, Beyo\u{g}lu, \.{I}stanbul, Turkey}
\curraddr{}
\email{bbektas@fsm.edu.tr}

\author{Shunya Fujii}
\address{Department of Mathematics,
 Shimane University, Nishikawatsu 1060 Matsue, 690-8504, Japan.}
\curraddr{}
\email{shunyaf3@gmail.com}

\author{Shun Maeta}
\address{Department of Mathematics,
Faculty of Education,
Chiba University, 1-33, Yayoicho, Inage-ku, Chiba-shi, Chiba, 263-8522 Japan.}
\curraddr{}
\email{shun.maeta@gmail.com~{\em or}~shun.maeta@chiba-u.jp}
\thanks{The third author is partially supported by the Grant-in-Aid for Young Scientists, No.19K14534, Japan Society for the Promotion of Science.}

\subjclass[2010]{53C25, 53C40, 53C42}

\date{}

\dedicatory{}

\commby{}

\keywords{Yamabe solitons; conformal solitons; pseudo-Euclidean spaces}
\begin{abstract}
In this paper, we completely classify conformal solitons on pseudo--Riemannian hypersurfaces in pseudo--Euclidean spaces 
arisen from the position vector field.
In particular, the classification of Yamabe solitons on pseudo--Riemannian hypersurfaces in pseudo--Euclidean spaces arisen from the position vector field can be obtained.
\end{abstract}

\maketitle


\bibliographystyle{amsplain}

\section{Introduction}\label{intro} 
Soliton solutions have been developed rapidly in recent years as important object in geometric flows. In fact, they appear as possible singularity models. 
Therefore, one of the most interesting problem is the classification problem of them.
For Ricci solitons, that is, the special solutions of the Ricci flow. S. Brendle \cite{Brendle} brought significant progress. He showed that "any 3-dimensional complete noncompact $\kappa$-noncollapsed gradient steady Ricci soliton with positive curvature is the Bryant soliton" which is a famous conjecture of G. Perelman \cite{Perelman1}.

For Yamabe solitons, that is, the special solutions of the Yamabe flow, there are many classification results under some assumptions of the curvatures (cf. \cite{CSZ12}, \cite{CMM12}, \cite{2}).
Recently, the third author classified nontrivial 3-dimensional complete gradient Yamabe solitons \cite{Maeta25}.
To understand the Yamabe soliton, many generalizations of it have been introduced.
For example, almost Yamabe solitons \cite{BB13} (see also \cite{SM19}), gradient $k$-Yamabe solitons \cite{CMM12}, $h$-almost gradient Yamabe solitons \cite{Zeng20} have been introduced. 
The most general notion of Yamabe solitons might be conformal solitons \cite{FM21} (see also \cite{Yano40}, \cite{Tashiro65}):
\begin{defn}[\cite{FM21}]
A Riemannian manifold $(M,g)$ is called a {\em conformal soliton} if there exists a complete vector field $v$ such that 
\begin{equation}\label{CS}
\varphi g=\frac{1}{2}\mathcal{L}_vg,
\end{equation} 
where $\varphi:M\rightarrow \mathbb{R}$ is a smooth function and $\mathcal{L}_v g$ 
is the Lie derivative of the metric tensor $g$ with respect to $v$. 
If $v\equiv0$, then the conformal soliton is called trivial.
\end{defn}
J. Cheeger and T. H. Colding studied conformal gradient solitons and gave a characterization theorem \cite{CC96}. 
Recently, the second and the third authors classified conformal solitons on hypersurfaces in Euclidean spaces under some assumptions \cite{FM21}.

Although the notion of conformal solitons was given and studied in a Riemannian manifold, the conformal soliton equation 
\eqref{CS} can be defined in a pseudo--Riemannian manifold as follows.

\begin{defn}
A pseudo--Riemannian manifold $(M,g)$ is called a {\em conformal soliton} if there exists a complete vector field $v$ such that 
\begin{equation}\label{CSn}
\varphi g=\frac{1}{2}\mathcal{L}_vg,
\end{equation} 
where $\varphi:M\rightarrow \mathbb{R}$ is a smooth function and $\mathcal{L}_v g$ 
is the Lie derivative of the metric tensor $g$ with respect to $v$. 
We denote the conformal soliton by $(M,g,v,\varphi)$.
If $v\equiv0$, then $(M,g,v,\varphi)$ is called trivial.
If  $v$ is the gradient of some smooth function $f$ on $M$, then $(M,g,f)$ is called a conformal gradient soliton. 
For smooth functions $f$ and $\varphi$, we have $\varphi g=\nabla\nabla f.$
\end{defn}

\begin{rem}
Conformal solitons include Yamabe solitons, almost Yamabe solitons, gradient $k$-Yamabe solitons, $h$-almost Yamabe solitons and conformal gradient solitons. Therefore, {\bf all the results in this paper can be applied to all these solitons.}

\end{rem}

In \cite{BCGG}, M. Brozos-V\'{a}zquez et. al. studied three dimensional Lorentzian homogeneous Ricci
soliton and the first author \cite{BBD} also classified the Ricci soliton on a pseudo--Riemannian 
hypersurface of a Minkowski space $\mathbb{E}^4_1$ whose the potential vector field 
is the tangential part of the position vector.

The position vector $V$ of a pseudo--Riemannian submanifold in the pseudo--Euclidean space
is the most elementary and geometric object, see \cite{Chen3, Chen4}. 
A. Fialkow \cite{Fialkow} introduced the notion of concircular vector field $v$ on a Riemannian 
manifold $M$ as vector field which satisfy 
\begin{equation}
\nabla_X v=\mu X,\;\;X\in TM
\end{equation}
where $\nabla$ denotes the Levi--Civita connection of $M$, 
$TM$ is the tangent bundle of $M$ 
and $\mu$ is a nontrivial function on $M$. 
As a particular case for the function $\mu=1$, a concircular vector 
field $v$ is called a concurrent vector field. The notion of 
concircular vector fields on Riemannian manifolds can be extended naturally to concircular vector fields
on pseudo--Riemannian manifolds. 
There are several studies of concurrent vector fields (see for example \cite{BY74}, \cite{5} and \cite{6}). 
The position vector field $V$ on a pseudo-Euclidean space $\mathbb{E}_s^m$ satisfies $\nabla_X V=X$, for any vector field $X$. Therefore, the position vector $V$ is the best known example as concurrent vector field. 

In this paper, we completely classify conformal solitons on a pseudo--Riemannian hypersurface in a pseudo-Euclidean space 
$\mathbb{E}^{n+1}_s$ arisen from the position vector field $V$. 

\begin{thm}\label{main}
Any conformal soliton $(M,g,V^T,\varphi)$ on a pseudo--Riemannian hypersurface $M$
 in a pseudo-Euclidean space 
$\mathbb{E}^{n+1}_s$ is contained in a hyperplane, a conic hypersurface, a pseudo hyperbolic space, or a pseudo sphere.
\end{thm}

As a corollary, we completely classify Yamabe solitons on a pseudo--Riemannian hypersurface in a pseudo-Euclidean space arisen from the position vector field.

\begin{cor}
Any Yamabe soliton $(M,g,V^T,\varphi)$ on a pseudo--Riemannian hypersurface $M$
 in a pseudo-Euclidean space 
$\mathbb{E}^{n+1}_s$ is contained in a hyperplane, a conic hypersurface, a pseudo hyperbolic space, or a pseudo sphere.
\end{cor}

\section{Preliminaries}\label{Pre} 

Let $(N,\tilde{g})$ be an $m$-dimensional pseudo-Riemannian manifold and $(M,g)$ be an $n$-dimensional submanifold in $(N,\tilde{g})$.
All manifolds in this paper are assumed to be smooth, orientable and connected.
 We denote Levi-Civita connections on $(M,g)$ and $(N,\tilde{g})$ by $\nabla$ and $\tilde{\nabla}$, respectively.
The Lie derivative of $g$ is defined by 
$$\left(\mathcal{L}_Xg\right)(Y,Z)=X(g(Y,Z))-g([X,Y],Z)-g(Y,[X,Z]),$$
for any vector fields $X,Y, Z$ on $M$.

For any vector fields $X,Y$ tangent to $M$ and $\eta$ normal to $M$, the formula of Gauss is given by
\begin{equation*}
{\widetilde{\nabla}}_XY={\nabla}_XY+h(X,Y),
\end{equation*}
where ${\nabla}_XY$ and $h(X,Y)$ are the tangential and the normal components of ${\tilde{\nabla}}_XY$.
The formula of Weingarten is given by
\begin{equation*}
{\widetilde{\nabla}}_X\eta=-A_{\eta}(X)+D_X\eta,
\end{equation*}
where $-A_{\eta}(X)$ and $D_X\eta$ are the tangential and the normal components of ${\widetilde{\nabla}}_X\eta$.
$A_{\eta}(X)$ and $h(X,Y)$ are related by
\begin{equation*}
g(A_{\eta}(X),Y)=\tilde{g}(h(X,Y),\eta).
\end{equation*}
The mean curvature vector $H$ of $M$ in $N$ is given by
\begin{equation*}
\displaystyle H(p)=\frac{1}{n}~\sum_{i=1}^n\epsilon_ih(e_i,e_i),
\end{equation*}
where $\{e_i\}_{i=1}^n$ is any orthonormal frame on $M$ at $p$, and $\epsilon_i=g(e_i,e_i).$ For any vector fields $X,Y,Z,W$ tangent to $M$, the equation of Gauss is given by
\begin{equation*}
\begin{tabular}{ll}
$\tilde{g}(\widetilde{Rm}(X,Y)Z,W)=$ & $g(Rm(X,Y)Z,W)$ \vspace{0.3pc}\\
~ & $+\tilde{g}(h(X,Z),h(Y,W))$ \vspace{0.3pc}\\
~ & $-\tilde{g}(h(X,W),h(Y,Z)),$
\end{tabular}
\end{equation*}
where $Rm$ and $\widetilde{Rm}$ are Riemannian curvature tensors of $M$ and $N$, respectively.
The equation of Codazzi is given by
\begin{equation*}
(\widetilde{Rm}(X,Y)Z)^{\perp}=({\bar{\nabla}}_Xh)(Y,Z)-({\bar{\nabla}}_Yh)(X,Z),
\end{equation*}
where $(\widetilde{Rm}(X,Y)Z)^{\perp}$ is the normal component of $\widetilde{Rm}(X,Y)Z$ and ${\bar{\nabla}}_Xh$ is defined by
\begin{equation*}
({\bar{\nabla}}_Xh)(Y,Z)=D_Xh(Y,Z)-h({\nabla}_XY,Z)-h(Y,{\nabla}_XZ).
\end{equation*}
If $N$ is a space of constant curvature, then the equation of Codazzi reduces to
\begin{equation*}
0=({\bar{\nabla}}_Xh)(Y,Z)-({\bar{\nabla}}_Yh)(X,Z).
\end{equation*}


Let $\mathbb{E}^{n+1}_s$ be an $(n+1)$-dimensional pseudo--Euclidean space with the flat metric $\tilde{g}$
given as follows
$$\tilde g=-dx_1^2-\cdots-dx_s^2+dx_{s+1}^2+\cdots+dx_{n+1}^2,$$
where $(x_1,\cdots,x_s,x_{s+1},\cdots,x_{n+1})$ denotes the usual coordinates in $\mathbb{R}^{n+1}$.
The pseudo--Riemannian space forms in $\mathbb{E}^{n+1}_s$ are defined by 
\begin{align}
\mathbb{S}^n_s(c^2)&=\{{\bf x}\in\mathbb{E}^{n+1}_s\;|\;\tilde{g}({\bf x},{\bf x})=c^{-2}\}\\
\mathbb{H}^n_{s-1}(-c^2)&=\{{\bf x}\in\mathbb{E}^{n+1}_{s}\;|\;\tilde{g}({\bf x},{\bf x})=-c^{-2}\}.
\end{align}
These spaces are complete and of constant curvature $c^2$ and $-c^2$, respectively. 
$\mathbb{S}^n_s(c^2)$ and $\mathbb{H}^n_{s-1}(-c^2)$ are called a pseudo--sphere and a pseudo--hyperbolic 
space, respectively.
\begin{rem}
\label{remshaope}
The shape operator of the pseudo--Riemannian hypersurface $M$ in the Lorentz--Minkowski space 
$\mathbb{E}^{n+1}_1$ can be expressed 
in one of the following types (cf. \cite{O'Neill83} P.261, see also \cite{Lucas11}), 
$$(1)~~
\left(
\begin{array}{cccc}
\kappa_1 & & & \text{\huge{0}}\\
 &  \kappa_2 & & \\
 & & \ddots & \\
\text{\huge{0}} & & & \kappa_n
\end{array}
\right),
\qquad
(2)~~
\left(
\begin{array}{cc|ccc}
\kappa_1 & -a & & & \text{\huge{0}}\\
 a & \kappa_1 & & &  \\ \hline
   & & \kappa_3 & &  \\
   & & & \ddots & \\
 \text{\huge{0}} & & & & \kappa_n
\end{array}
\right)
~~~a\not=0,
$$
for an orthonormal frame $\{E_1,E_2,\cdots,E_{n+1}\}$;
\begin{align*}
&\tilde g(E_1,E_1)=-1,\qquad \tilde g(E_1,E_i)=0,\quad i=2,3,\cdots,n+1,\\
&\tilde g(E_i,E_j)=\delta_{ij},\quad 2\leq i,j\leq n+1,
\end{align*}
and 
$$(3)~~
\left(
\begin{array}{cc|ccc}
\kappa_1 & 0 & & & \text{\huge{0}}\\
 1 & \kappa_1 & & &  \\ \hline
   & & \kappa_3 & &  \\
   & & & \ddots & \\
 \text{\huge{0}} & & & & \kappa_n
\end{array}
\right),
\qquad
(2)~~
\left(
\begin{array}{ccc|ccc}
\kappa_1 & 0 & 0 & & & \\
 0 & \kappa_1 & 1 & & \text{\huge{0}} & \\ 
 -1 & 0 & \kappa_1 & & & \\ \hline
  & & & \kappa_4 & &  \\
  & \text{\huge{0}} & & & \ddots & \\
  & & & & & \kappa_n
\end{array}
\right),
$$
for a pseudo-orthonormal frame $\{E_1,E_2,\cdots,E_{n+1}\}$;
\begin{align*}
&\tilde g(E_1,E_2)=-1,\qquad \tilde g(E_1,E_1)=\tilde g(E_2,E_2)=0,\\
&\tilde g(E_i,E_j)=0,\quad i=1,2,~j=3,\cdots,n+1,\\
&\tilde g(E_i,E_j)=\delta_{ij},\quad 3\leq i,j\leq n+1.
\end{align*}
\end{rem}
\section{Conformal solitons with a concurrent vector field}\label{CSwcv}


In this section, we consider a conformal soliton with a concurrent vector field.

\begin{prop}\label{AYSwithC}
Any conformal soliton $(M,g,v,\varphi)$ which has a concurrent vector field $v$ is a conformal gradient soliton with 
$\varphi =1$.
\end{prop}

\begin{proof}
Since $v$ is a concurrent vector field, we have
\begin{equation}\label{c.0}
g(v,X) =g(v,{\nabla}_Xv)=X(\frac{1}{2}g(v,v)),
\end{equation}
and
\begin{equation}\label{c.1}
\begin{tabular}{ll}
$\left(\mathcal{L}_vg\right)(X,Y)$ & = $vg(X,Y) - g([v,X],Y) - g(X,[v,Y])$ \vspace{0.3pc}\\
~& $= vg(X,Y) - vg(X,Y) + g({\nabla}_Xv,Y) + g(X,{\nabla}_Yv)$ \vspace{0.3pc}\\
~& $= 2g(X,Y),$
\end{tabular}
\end{equation}
for any vector fields $X$, $Y$ on $M$.
By putting $f = \frac{1}{2}g(v,v)$ on the equation $(\ref{c.0})$, we obtain $ v = \nabla f $.
Substituting $(\ref{c.1})$ into $(\ref{CS})$, we have
\begin{equation*}
\varphi = 1.
\end{equation*}
\end{proof}

\section{Conformal solitons on pseudo--Riemannian submanifolds}\label{SubCS}

In this section, we assume that $(N,\tilde{g})$ is a pseudo--Riemannian manifold endowed with a concurrent vector field $V$ and $(M,g)$ is a pseudo--Riemannian submanifold in $(N,\tilde{g})$.
$V^T$ and $V^{\perp}$ denote the tangential and the normal components of $V$, respectively.

Firstly, we show the following lemma which will be used later for the purpose of classification of the conformal solitons. 

\begin{lem}\label{NSAYS}
Any conformal soliton $(M,g,V^T,\varphi)$ on a pseudo--Riemannian submanifold $M$ in $N$ satisfies
\begin{equation}\label{ENSAYS}
(\varphi-1)g(X,Y)=g(A_{V^{\perp}}(X),Y),
\end{equation}
for any vector fields $X, Y$ on $M$.
\end{lem}

\begin{proof}
From the definition of the Lie derivative, we have
\begin{equation}\label{s.3}
\begin{tabular}{ll}
$(\mathcal{L}_{V^T}g)(X,Y)$ & = $V^Tg(X,Y) - g(\nabla_{V^T}X-\nabla_X{V^T},Y) - g(X,\nabla_{V^T}Y-\nabla_Y{V^T})$\vspace{0.5pc}\\
~ & $=g(\nabla_X{V^T},Y) + g(X,\nabla_Y{V^T})$\vspace{0.5pc}\\
~ & $=\tilde g(\tilde{\nabla}_X{V}-\tilde{\nabla}_X{V^{\perp}},Y) + \tilde g(X,\tilde{\nabla}_Y{V}-\tilde{\nabla}_Y{V^{\perp}})$\vspace{0.5pc}\\
~ & $=2g(X,Y)+2g(A_{V^{\perp}}(X),Y)$,
\end{tabular}
\end{equation}
for any vector fields $X, Y$ on $M$.
Combining $(\ref{s.3})$ with $(\ref{CSn})$, we obtain $(\ref{ENSAYS})$.
\end{proof}

\begin{prop}\label{CS is GCS}
Any conformal soliton $(M,g,V^T,\varphi)$ on a pseudo--Riemannian submanifold $M$ in $N$ 
is a conformal gradient soliton. 
\end{prop}

\begin{proof}
Set
\begin{equation*}
\displaystyle f=\frac{1}{2}~\tilde{g}(V,V).
\end{equation*}
For any vector field $X$ on $M$, we obtain
$$g(V^T,X)=\tilde g(V,X)=\tilde g (V,\tilde \nabla_XV)=X(\frac{1}{2}\tilde g(V,V))=Xf=g(\nabla f, X).$$
\end{proof}

\begin{prop}\label{MCS}
If a conformal soliton $(M,g,V^T,\varphi)$ on a pseudo--Riemannian submanifold $M$ in $N$ is minimal, 
then $\varphi=1$.
\end{prop}

\begin{proof}
Let $\{e_1,  \cdots , e_n\}$ be an orthonormal frame on $M$. By Lemma~$\ref{NSAYS}$, we have
\begin{equation*}
(\varphi-1)g_{ij}=g(A_{V^{\perp}}(e_i),e_j)=\tilde{g}(h(e_i,e_j),V^{\perp}).
\end{equation*}
Since $M$ is minimal and taking the trace, we obtain
\begin{equation*}
n(\varphi-1)=n\tilde g(H,V^{\perp})=0. 
\end{equation*}
Therefore, we conclude that
\begin{equation*}
\varphi=1.
\end{equation*}
\end{proof}

\section{Proof of Theorem~$\ref{main}$}\label{Classification}

We hereafter denote $V$ by the position vector field of a pseudo--Riemannian hypersurface  
in a pseudo-Euclidean space $\mathbb{E}^{n+1}_s$.
In this section, we give the proof of Theorem~$\ref{main}$.


To show Theorem \ref{main}, we first consider the case that $V=V^T.$  B. Y. Chen and Y. M. Oh \cite{CO17} showed the following proposition.

\begin{prop}[\cite{CO17}]\label{Chenlem}

Let $M$ be a pseudo--Riemannian submanifold in a pseudo-Euclidean space $\mathbb{E}^m_s$. If the position vector field $V$ of $M$ in $\mathbb{E}^m_s$ is either space-like or time-like, then $V=V^T$ holds identically if and only if $M$ is a conic submanifolds with the vertex at the origin.
\end{prop}

We can relax the assumption, and get the following.

\begin{prop}\label{conic}
Let $M$ be a pseudo-Riemannian submanifold in a pseudo-Euclidean space $\mathbb{E}^m_s$.
Then, $V=V^T$ holds identically if and only if $M$ is a conic submanifolds with the vertex at the origin.
\end{prop}

\begin{proof}

For the position vector field $V=\displaystyle \sum_{i=1}^mx_i\frac{\partial}{\partial x_i}$, let $k=\sqrt{\displaystyle \sum_{i=1}^{m}x_i^2 }$. If $ V=V^T$ on $M$, $e=\frac{1}{k}V$ is a vector field tangent to $M$. Then we get
$$e=\tilde{\nabla}_eV=\tilde{\nabla}_e(ke)=e(k)e+k\tilde{\nabla}_ee,$$
and
$$e(k)=(\displaystyle \sum_{i=1}^{m}\frac{x_i}{k}\frac{\partial}{\partial x_i})(k)=\displaystyle \sum_{i=1}^{m}(\frac{x_i}{k}\cdot \frac{1}{2}\cdot \frac{2x_i}{k})=1.$$
So we conclude $\tilde{\nabla}_ee=0$. 
Therefore, by the same argument as in the proof of Proposition \ref{Chenlem}, one can get the proposition.
\end{proof}

\begin{proof}[Proof of Theorem $\ref{main}$]
Let $\alpha$ be a mean curvature and $\lambda$ be a support function of $M$,  i.e., $H=\alpha N$ and $\lambda =\tilde{g}(N,V)$ with a unit normal vector field $N$. Set $\epsilon_N=\tilde{g}(N,N)=\pm 1$ and $U_0=\{x\in M| \lambda=0\}$.
 
Let $f$ be a smooth function on $M$ defined by 
\begin{equation*}
f(x)\coloneqq \det A(x) \hspace{0.2in} x\in M.
\end{equation*}  
 
\noindent
\underline{Case 1. $U_0 = \emptyset$}:
If $V^\perp\not=0$, 
Lemma~$\ref{NSAYS}$ implies that the pseudo--Riemannian hypersurface $M$ 
must have a diagonalizable shape operator with respect to orthonormal frame $e_i$ for $i=1,\cdots ,n$ .
From Lemma~$\ref{NSAYS}$, we have
\begin{equation*}
(\varphi-1)g_{ij}=\tilde{g}(h(e_i,e_j),V^{\perp})=\tilde{g}(\epsilon_N{\kappa}_i g_{ij} N,V)=\epsilon_N{\kappa}_i g_{ij} \lambda ,
\end{equation*}
where $A_N(e_i)={\kappa}_ie_i, ~ (i=1,\cdots ,n)$. 
Hence we have
\begin{equation}\label{e.1}
\varphi-1=\epsilon_N\lambda {\kappa}_i.
\end{equation}
By taking the summation, we obtain
\begin{equation}\label{e.2}
\varphi-1=\lambda \alpha.
\end{equation} 
Comparing ~$(\ref{e.1})$ and ~$(\ref{e.2})$, we have
\begin{equation*}
{\kappa}_i =\epsilon_N\alpha.
\end{equation*}
Thus $M$ is totally umbilical with $A_N(e_i)=\epsilon_N\alpha e_i$ and $h$ satisfies $h(X,Y)=\alpha g(X,Y) N$.
Since $N$ is a unit normal vector field, we have
\begin{equation*}
0={\widetilde{\nabla}}_X(\tilde{g}(N,N))=2\tilde{g}({\widetilde{\nabla}}_XN,N)=2\tilde{g}(D_XN,N).
\end{equation*}
Therefore, $D_XN=0$.
Hence we obtain
\begin{equation*}
\begin{tabular}{rl}
$({\bar{\nabla}}_Xh)(Y,Z)=$ & $D_Xh(Y,Z)-h({\nabla}_XY,Z)-h(Y,{\nabla}_XZ)$ \vspace{0.5pc} \\
~$=$ & $X(\alpha ) g(Y,Z) N,$
\end{tabular}
\end{equation*}
for any vector fields $X, Y, Z$ on $M$.
From the equation of Codazzi, we have
\begin{equation*}
X(\alpha)Y=Y(\alpha)X.
\end{equation*}
Since we can assume that $X$ and $Y$ are linearly independent, we conclude that $\alpha$ and $f$ are constant respectively.

If $\alpha=0$, then by ${\widetilde{\nabla}}_XN=0,$ $N$, restricted to $M$, is a constant vector field in $\mathbb{E}^{n+1}_s$
and we have
\begin{equation*}
{\widetilde{\nabla}}_X(\tilde{g}(V,N))=\tilde{g}({\widetilde{\nabla}}_XV,N)+\tilde{g}(V,{\widetilde{\nabla}}_XN)=\tilde{g}(X,N)=0.
\end{equation*}
This shows that $\lambda=\tilde{g}(V,N)$ is a nonzero constant when $V$ and $N$ are restricted to $M$. 
Therefore, $M$ is contained in a hyperplane normal to $N$ which does not through the origin and $\varphi=1$.

If $\alpha\not=0$, then we have
\begin{equation*}
{\widetilde{\nabla}}_X(V+\epsilon_N{\alpha}^{-1}N)=X+\epsilon_N{\alpha}^{-1} {\tilde{\nabla}}_XN=X+\epsilon_N{\alpha}^{-1} (-A_N(X))=0.
\end{equation*}
This shows that the vector field $V+\epsilon_N{\alpha}^{-1}N$, restricted to $M$, is a constant one in $\mathbb{E}^{n+1}_s.$ 
Therefore, $M$ is contained in a pseudo hyperbolic space $\mathbb{H}^{n}_{s-1}$ or a pseudo sphere 
$\mathbb{S}^{n}_{s}$ of $\mathbb{E}^{n+1}_s$.

\noindent
\underline{Case 2. $U_0=M$}:

We have $V=V^T$. 
By Proposition~\ref{conic}, we obtain that $M$ is contained in a conic hypersurface.

\noindent
\underline{Case 3. Others}:

Take $p\in M\backslash U_0$, that is, $\lambda\not=0$ on some open set $\Omega\ni p$.
By the same argument as in Case 1, we have $\Omega$ is an open portion of a hyperplane, a pseudo hyperbolic space $\mathbb{H}^{n}_{s-1}$ or a pseudo sphere 
$\mathbb{S}^{n}_{s}$ of $\mathbb{E}^{n+1}_s$.

We consider the case that $\Omega$ is an open portion of a hyperplane.
Without loss of generality, we can take $\Omega$ as the maximum connected component which is an open set including $p$ on $M\backslash U_0$.
On $\Omega$, $\lambda=\tilde g(V,N)(\not=0)$ is constant, say $\lambda_\Omega$.
Since $M$ is connected, if $\Omega$ is closed, then $\Omega=M$, which is a contradiction.
If $M$ is not closed, then we have
$\partial \Omega\not=\emptyset$ and $\partial \Omega\cap\Omega=\emptyset.$
Take $q\in \partial \Omega$. Since $\lambda$ is continuous, we have $\lambda(q)=\lambda_\Omega$. Thus we can take an open neighborhood $U_q$ of $q$ such that $\lambda\not=0$ on $U_q$. Since $\Omega$ is the maximum connected component, we have a contradiction.
Hence, we have that $\Omega$ is an open portion of a pseudo hyperbolic space $\mathbb{H}^{n}_{s-1}$ or a pseudo sphere 
$\mathbb{S}^{n}_{s}$ of $\mathbb{E}^{n+1}_s$. 

Assume that Int$U_0 \neq \emptyset$.
On $U_0$, we have $V=V^T$. Since $V$ is a concurrent vector field,
\begin{equation*}
X =\widetilde{\nabla}_XV^T = \nabla_XV^T+h(X,V^T),
\end{equation*}
on Int$U_0$. So we obtain
$h(X,V^T)=0.$
Therefore, we have $A_N(V^T)=0$. Thus, the shape operator has zero eigenvalue. Hence, $f=\det A=0$ on Int$U_0$.
From this, $X(f) = 0$ on Int$U_0$ for any vector field $X$ on $M$. From Case 1, $X(f) = 0$ on $M\backslash U_0$. This means $X(f) = 0$ on $M$. So we conclude that $f$ is constant on $M$. But $f\neq 0$ on $M\backslash U_0$, this is a contradiction. Therefore, Int$U_0 = \emptyset$.

By the same argument as in Case 1, $\kappa_i = \kappa_j$ and $X(\kappa_i) = 0$ on $M\backslash U_0$ for any $1\leq i, j\leq n$ and  any vector field $X$. From this and Int$U_0 = \emptyset$, $\kappa_i$ is constant which doesn't depend on $i$ on $M$. 

Therefore, $M$ is contained in a pseudo hyperbolic space $\mathbb{H}^{n}_{s-1}$ or a pseudo sphere 
$\mathbb{S}^{n}_{s}$ of $\mathbb{E}^{n+1}_s$.
\end{proof}



\bibliographystyle{amsbook}

\end{document}